\documentclass{amsart}
\usepackage{amsfonts}
\usepackage{hyperref}

\newtheorem{theorem}{Theorem}
\theoremstyle{plain}

\newtheorem{corollary}{Corollary}

\newtheorem{lemma}{Lemma}

\newtheorem{remark}{Remark}

\numberwithin{equation}{section}
\input{tcilatex}

\begin{document}
\title[New sharp Jordan type inequalities]{New sharp Jordan type
inequalities and their applications}
\author{Zhen-Hang Yang}
\address{System Division, Zhejiang Province Electric Power Test and Research
Institute, Hangzhou, Zhejiang, China, 310014}
\email{yzhkm@163.com}
\date{April 10, 2012}
\thanks{This paper is in final form and no version of it will be submitted
for publication elsewhere.}

\begin{abstract}
In this paper, we prove that for $x\in \left( 0,\pi /2\right) $ 
\begin{equation*}
\left( \cos p_{0}x\right) ^{1/p_{0}}<\frac{\sin x}{x}<\left( \cos \tfrac{x}{3%
}\right) ^{3}
\end{equation*}%
with the best constants $p_{0}=0.347307245464...$ and $1/3$. Moreover, if $%
p\in $ $(0,1/3]$ then the double inequality 
\begin{equation*}
\beta _{p}\left( \cos px\right) ^{1/p}<\frac{\sin x}{x}<\left( \cos
px\right) ^{1/p}
\end{equation*}%
holds for $x\in \left( 0,\pi /2\right) $, where $\beta _{p}=2\pi ^{-1}\left(
\cos \frac{p\pi }{2}\right) ^{-1/p}$ and $1$ are the best possible. Its
reverse one holds if $p\in \left[ 1/2,1\right] $. As applications, some new
inequalities are established.
\end{abstract}

\maketitle

\section{Introduction}

The classical Jordan's inequality \cite{Mitrinovic.Springer.1970} states
that for $x\in (0,\pi /2]$ 
\begin{equation}
\frac{2}{\pi }\leq \frac{\sin x}{x}\leq 1,  \label{Jordan}
\end{equation}%
with equality holds if and only if $x=\pi /2$.

Some new developments on refinements, generalizations and applications of
Jordan's inequality can be found in \cite{Qi.JIA.2009} and related
references therein.

In \cite{Qi.MIA.2(4)(1999)}, Qi et al. proved that 
\begin{equation}
\cos ^{2}\frac{x}{2}<\frac{\sin x}{x}\text{ }  \label{Qi}
\end{equation}%
holds for $x\in \left( 0,\pi /2\right) $. Recently, Kl\'{e}n et al. \cite[%
Theorem 2.4]{Klen.JIA.2010} showed that the function $p\mapsto \left( \cos
px\right) ^{1/p}$ is decreasing on $\left( 0,1\right) $ and proved that for $%
x\in \left( -\sqrt{27/5},\sqrt{27/5}\right) $ \emph{\ }%
\begin{equation}
\cos ^{2}\frac{x}{2}\leq \frac{\sin x}{x}\leq \cos ^{3}\frac{x}{3}\leq \frac{%
2+\cos x}{3}  \label{Klen}
\end{equation}%
hold.

This paper is motivated by these studies and it is aimed at giving the sharp
bounds $\left( \cos px\right) ^{1/p}$ for $\left( \sin x\right) /x$, that
is, for $x\in \left( 0,\pi /2\right) $, determine the best $p,q\in (0,1]$
such that 
\begin{equation}
\left( \cos px\right) ^{1/p}<\frac{\sin x}{x}<\left( \cos qx\right) ^{1/q}
\label{M}
\end{equation}%
hold, respectively.

Our main results are contained the following two theorems.

\begin{theorem}
\label{Theorem Ma}Let $p,q\in \left( 0,1\right] $. Then the inequalities 
\begin{equation}
\left( \cos px\right) ^{1/p}<\frac{\sin x}{x}<\left( \cos qx\right) ^{1/q}
\label{Ma}
\end{equation}%
hold for $x\in \left( 0,\pi /2\right) $ if and only if $p\in \left[ p_{0},1%
\right] $ and $q\in (0,1/3]$, where $p_{0}=0.347307245464...$.

Moreover, we have%
\begin{eqnarray}
\left( \cos \frac{x}{3}\right) ^{\alpha } &<&\frac{\sin x}{x}<\left( \cos 
\frac{x}{3}\right) ^{3},  \label{Ma1} \\
\left( \cos p_{0}x\right) ^{1/p_{0}} &<&\frac{\sin x}{x}<\left( \cos
p_{0}x\right) ^{1/\left( 3p_{0}^{2}\right) },  \label{Ma2}
\end{eqnarray}%
where the exponents $\alpha =\allowbreak \frac{2\left( \ln \pi -\ln 2\right) 
}{\ln 4-\ln 3}$, $3$, $1/p_{0}$, $1/\left( 3p_{0}^{2}\right) $ are the best
constants.
\end{theorem}

\begin{theorem}
\label{Theorem Mb}For $p\in (0,1]$, let $f_{p}$ be the function defined on $%
\left( 0,\pi /2\right) $ by%
\begin{equation}
f_{p}\left( x\right) =\ln \frac{\sin x}{x}-\frac{1}{p}\ln \cos px.
\label{f_p}
\end{equation}%
Then $f_{p}$ is decreasing if $p\in $ $(0,1/3]$ and increasing if $p\in %
\left[ 1/2,1\right] $.

Moreover, if $p\in $ $(0,1/3]$ then 
\begin{equation}
\beta _{p}\left( \tfrac{\pi }{4}\right) \left( \cos px\right) ^{1/p}<\frac{%
\sin x}{x}<\left( \cos px\right) ^{1/p}  \label{Mb}
\end{equation}%
with the best possible constants $\beta _{p}\left( \pi /2\right) =2\pi
^{-1}\left( \cos \frac{p\pi }{2}\right) ^{-1/p}$ and $1$. (\ref{Mb}) is
reversed if $p\in \left[ 1/2,1\right] $.
\end{theorem}

\section{Lemmas}

\begin{lemma}[\protect\cite{Vamanamurthy.183.1994}, \protect\cite%
{Anderson.New York. 1997}]
\label{Lemma A}Let $f,g:\left[ a,b\right] \mapsto \mathbb{R}$ be two
continuous functions which are differentiable on $\left( a,b\right) $.
Further, let $g^{\prime }\neq 0$ on $\left( a,b\right) $. If $f^{\prime
}/g^{\prime }$ is increasing (or decreasing) on $\left( a,b\right) $, then
so are the functions 
\begin{equation*}
x\mapsto \frac{f\left( x\right) -f\left( b\right) }{g\left( x\right)
-g\left( b\right) }\text{ \ \ \ and \ \ \ }x\mapsto \frac{f\left( x\right)
-f\left( a\right) }{g\left( x\right) -g\left( a\right) }.
\end{equation*}
\end{lemma}

\begin{lemma}[\protect\cite{Biernacki.9.1955}]
\label{Lemma B}Let $a_{n}$ and $b_{n}$ $(n=0,1,2,...)$ be real numbers and
let the power series $A\left( t\right) =\sum_{n=1}^{\infty }a_{n}t^{n}$ and $%
B\left( t\right) =\sum_{n=1}^{\infty }b_{n}t^{n}$ be convergent for $|t|<R$.
If $b_{n}>0$ for $n=0,1,2,...$, and $a_{n}/b_{n}$ is strictly increasing (or
decreasing) for $n=0,1,2,...$, then the function $A\left( t\right) /B\left(
t\right) $ is strictly increasing (or decreasing) on $\left( 0,R\right) $.
\end{lemma}

\begin{lemma}[{\protect\cite[pp.227-229]{Handbook.math.1979}}]
\label{Lemma C}We have%
\begin{eqnarray}
\cot x &=&\frac{1}{x}-\sum_{n=1}^{\infty }\frac{2^{2n}}{\left( 2n\right) !}%
|B_{2n}|x^{2n-1}\text{, \ }|x|<\pi ,  \label{2.1} \\
\tan x &=&\sum_{n=1}^{\infty }\frac{2^{2n}-1}{\left( 2n\right) !}%
2^{2n}|B_{2n}|x^{2n-1}\text{, \ }|x|<\pi /2,  \label{2.2}
\end{eqnarray}%
where $B_{n}$ is the Bernoulli numbers.
\end{lemma}

\begin{lemma}
\label{Lemma D}Let $F_{p}$ be the function defined $\left( 0,\pi /2\right) $
by 
\begin{equation}
F_{p}\left( x\right) =\frac{\ln \frac{\sin x}{x}}{\ln \cos px}.  \label{2.3}
\end{equation}%
Then $F_{p}$ is strictly increasing on $\left( 0,\pi /2\right) $ if $p\in (0,%
\sqrt{5}/5]$ and decreasing on $\left( 0,\pi /2\right) $ if $p\in \left[
1/2,1\right] $. Moreover, we have%
\begin{equation}
\tfrac{\ln 2-\ln \pi }{\ln \left( \cos \frac{1}{2}\pi p\right) }\ln \cos
px<\ln \frac{\sin x}{x}<\tfrac{1}{3p^{2}}\ln \cos px  \label{Main}
\end{equation}%
if $p\in (0,\sqrt{5}/5]$. The inequalities (\ref{Main}) are reversed if $%
p\in \lbrack 1/2,1]$.
\end{lemma}

\begin{proof}
For $x\in \left( 0,\pi /2\right) $, we define $f\left( x\right) =\ln \frac{%
\sin x}{x}$ and $g\left( x\right) =\ln \cos px$, where $p\in (0,1]$. Note
that $f\left( 0^{+}\right) =g\left( 0^{+}\right) =0$, then $F_{p}\left(
x\right) $ can be written as 
\begin{equation*}
F_{p}\left( x\right) =\frac{f\left( x\right) -f\left( 0^{+}\right) }{g\left(
x\right) -g\left( 0^{+}\right) }.
\end{equation*}%
Differentiation and using (\ref{2.1}) and (\ref{2.2}) yield 
\begin{equation*}
\frac{f^{\prime }\left( x\right) }{g^{\prime }\left( x\right) }=\allowbreak 
\frac{p\left( \frac{1}{x}-\cot x\right) }{\tan px}=\allowbreak \frac{%
\sum_{n=1}^{\infty }\frac{2^{2n}}{\left( 2n\right) !}|B_{2n}|x^{2n-1}}{%
\sum_{n=1}^{\infty }\frac{2^{2n}-1}{\left( 2n\right) !}%
p^{2n-2}2^{2n}|B_{2n}|x^{2n-1}}:=\allowbreak \frac{\sum_{n=1}^{\infty
}a_{n}x^{2n-1}}{\sum_{n=1}^{\infty }b_{n}x^{2n-1}},
\end{equation*}%
where 
\begin{equation*}
a_{n}=\frac{2^{2n}}{\left( 2n\right) !}|B_{2n}|\text{, \ }b_{n}=\frac{%
2^{2n}-1}{\left( 2n\right) !}p^{2n-2}2^{2n}|B_{2n}|\text{.}
\end{equation*}%
Clearly, if the monotonicity of $a_{n}/b_{n}$ is proved, then by Lemma \ref%
{Lemma B} it is deduced the monotonicity of $f^{\prime }/g^{\prime }$, and
then the monotonicity of the function $F_{p}$ easily follows from Lemma \ref%
{Lemma A}. Now we prove the monotonicity of $a_{n}/b_{n}$. Indeed,
elementary computation yields%
\begin{eqnarray*}
\frac{b_{n+1}}{a_{n+1}}-\frac{b_{n}}{a_{n}} &=&\left( 2^{2n+2}-1\right)
p^{2n}-\left( 2^{2n}-1\right) p^{2n-2} \\
&=&\left( 4^{n+1}-1\right) p^{2n-2}\left( p^{2}-\frac{1}{4}+\frac{3}{4\left(
4^{n+1}-1\right) }\right) ,
\end{eqnarray*}%
from which it is easy to obtain that for $n\in \mathbb{N}$ 
\begin{equation*}
\frac{b_{n+1}}{a_{n+1}}-\frac{b_{n}}{a_{n}}\left\{ 
\begin{array}{cc}
\leq 0 & \text{if }p^{2}<\frac{1}{5}, \\ 
> & \text{if }p^{2}\geq \frac{1}{4}.%
\end{array}%
\right.
\end{equation*}%
It is seen that $b_{n}/a_{n}$ is decreasing if $0<p\leq \sqrt{5}/5$ and
increasing if $1/2\leq p\leq 1$, which together with $a_{n}$, $b_{n}>0$ for $%
n\in \mathbb{N}$ leads to $a_{n}/b_{n}$ is strictly increasing if $0<p\leq 
\sqrt{5}/5$ and decreasing if $1/2\leq p\leq 1$.

By the monotonicity of the function $F_{p}$ and notice that 
\begin{equation*}
F_{p}\left( 0^{+}\right) =\tfrac{1}{3p^{2}}\text{ \ \ \ and \ \ \ }%
F_{p}\left( \tfrac{\pi }{2}^{-}\right) =\tfrac{\ln 2-\ln \pi }{\ln \left(
\cos \frac{1}{2}\pi p\right) },
\end{equation*}%
the inequalities (\ref{Main}) follow immediately.
\end{proof}

\begin{lemma}
\label{Lemma E}For $p\in \left( 0,1\right] $, let $f_{p}$ be the function
defined on $\left( 0,\pi /2\right) $ by \ref{f_p}).

(i) If $f_{p}\left( x\right) <0$ holds for all $x\in \left( 0,\pi /2\right) $
then $p\in (0,1/3]$.

(ii) If $f_{p}\left( x\right) >0$ for all $x\in \left( 0,\pi /2\right) $,
then $p\in \lbrack p_{0},1]$, where $p_{0}=0.347307245464...$ is the unique
root of equation%
\begin{equation}
f_{p}\left( \tfrac{\pi }{2}\right) \allowbreak =\ln \frac{2}{\pi }-\frac{1}{p%
}\ln \left( \cos \frac{1}{2}\pi p\right) =0  \label{f_p=0}
\end{equation}%
on $\left( 0,1\right] $.
\end{lemma}

\begin{proof}
Firstly, we show that there is a unique $p_{0}\in \left( 0,1\right) $ to
satisfy (\ref{f_p=0}) such that $f_{p}\left( \tfrac{\pi }{2}\right) <0$ for $%
p\in \left( 0,p_{0}\right) $ and $f_{p}\left( \tfrac{\pi }{2}\right) >0$ for 
$p\in (p_{0},1]$. Indeed, as mentioned previous (see \cite[Theorem 2.3]%
{Klen.JIA.2010}), the function $p\mapsto p^{-1}\ln \left( \cos \frac{1}{2}%
\pi p\right) $ is decreasing on $\left( 0,1\right) $, and therefore $%
p\mapsto f_{p}\left( \tfrac{\pi }{2}\right) $ is increasing on $\left(
0,1\right) $. Since 
\begin{eqnarray*}
f_{1/3}\left( \tfrac{\pi }{2}\right) &=&\allowbreak \ln \tfrac{2}{\pi }-3\ln 
\tfrac{\sqrt{3}}{2}<0, \\
f_{1/2}\left( \tfrac{\pi }{2}\right) &=&\allowbreak \ln \tfrac{2}{\pi }-2\ln 
\tfrac{\sqrt{2}}{2}>0,
\end{eqnarray*}%
so the equation ((\ref{f_p=0}) has a unique root $p_{0}$ on $\left(
0,1\right) $ and $p_{0}\in \left( 1/3,1/2\right) $ such that $f_{p}\left( 
\tfrac{\pi }{2}\right) <0$ for $p\in \left( 0,p_{0}\right) $ and $%
f_{p}\left( \tfrac{\pi }{2}\right) >0$ for $p\in (p_{0},1]$. Numerical
calculation reveals that $p_{0}=0.347307245464...$.

Secondly, if inequality $f_{p}\left( x\right) <0$ holds for $x\in \left(
0,\pi /2\right) $, then we have%
\begin{equation*}
\left\{ 
\begin{array}{l}
\lim_{x\rightarrow 0^{+}}\frac{f_{p}\left( x\right) }{x^{2}}%
=\lim_{x\rightarrow 0^{+}}\frac{\ln \frac{\sin x}{x}-\frac{1}{p}\ln \cos px}{%
x^{2}}=\allowbreak \frac{1}{2}p-\frac{1}{6}\leq 0, \\ 
f_{p}\left( \tfrac{\pi }{2}\right) =\allowbreak \ln \tfrac{2}{\pi }-\frac{1}{%
p}\ln \left( \cos \frac{1}{2}\pi p\right) \leq 0.%
\end{array}%
\right.
\end{equation*}%
Solving the inequalities for $p$ yields 
\begin{equation*}
p\in (0,1/3]\cap (0,p_{0}]=(0,1/3].
\end{equation*}%
In the same way, if inequality $f_{p}\left( x\right) >0$ for all $x\in
\left( 0,\pi /2\right) $, then 
\begin{equation*}
p\in \left[ 1/3,1\right] \cap \lbrack p_{0},1]=[p_{0},1],
\end{equation*}%
which completes the proof.
\end{proof}

\section{Proofs of Main Results}

\begin{proof}[Proof of Theorem \protect\ref{Theorem Ma}]
(i) We first prove the second inequality of (\ref{Ma}) holds if and only if $%
q\in (0,1/3]$.

The necessity follows from Lemma \ref{Lemma E}. It remains to show that the
condition $q\in (0,1/3]$ is sufficient. Since $q\in (0,1/3]\subset (0,\sqrt{5%
}/5]$, by the second inequality of (\ref{Main}) it is obtained that 
\begin{equation*}
\ln \frac{\sin x}{x}<\tfrac{1}{3q^{2}}\ln \cos qx=\tfrac{1}{3q}\ln \left(
\cos qx\right) ^{1/q}\leq \ln \left( \cos qx\right) ^{1/q},
\end{equation*}%
which proves the sufficiency.

(ii) Now we show that the first inequality of (\ref{Ma}) holds if and only
if $p\in \left[ p_{0},1\right] $.

The necessity is due to Lemma \ref{Lemma E}. We prove the first inequality
of (\ref{Ma}) holds if $p\in \left[ p_{0},1\right] $. In fact, in view of $%
p_{0}\in (0,\sqrt{5}/5]$ it follows from the first inequality of (\ref{Main}%
) that 
\begin{equation*}
\ln \frac{\sin x}{x}>\tfrac{\ln 2-\ln \pi }{\ln \left( \cos \frac{1}{2}\pi
p_{0}\right) }\ln \cos p_{0}x=\tfrac{p_{0}\left( \ln 2-\ln \pi \right) }{\ln
\left( \cos \frac{1}{2}\pi p_{0}\right) }\ln \left( \cos p_{0}x\right)
^{1/p_{0}}=\ln \left( \cos p_{0}x\right) ^{1/p_{0}},
\end{equation*}%
where the last equality holds is due to $p_{0}$ is the unique root of
equation (\ref{f_p=0}). And, by the monotonicity of the function $%
p\rightarrow \left( \cos px\right) ^{1/p}$, we conclude that for $p\in
\lbrack p_{0},1]$ 
\begin{equation*}
\frac{\sin x}{x}>\left( \cos p_{0}x\right) ^{1/p_{0}}\geq \left( \cos
px\right) ^{1/p},
\end{equation*}%
which implies the sufficiency.

(iii) Lastly, put $p=1/3$ and $p_{0}$ in (\ref{Main}) lead to (\ref{Ma1})
and (\ref{Ma2}), respectively.
\end{proof}

\begin{proof}[Proof of Theorem \protect\ref{Theorem Mb}]
Differentiation and using (\ref{2.1}) and (\ref{2.2}) yield$\allowbreak $ 
\begin{eqnarray*}
f_{p}^{\prime }\left( x\right) &=&\allowbreak \left( \cot x-\tfrac{1}{x}%
\right) +\tan px \\
&=&-\sum_{n=1}^{\infty }\frac{2^{2n}}{\left( 2n\right) !}|B_{2n}|x^{2n-1}+%
\sum_{n=1}^{\infty }\frac{2^{2n}-1}{\left( 2n\right) !}%
p^{2n-1}2^{2n}|B_{2n}|x^{2n-1} \\
&=&\sum_{n=1}^{\infty }\frac{\left( 2^{2n}-1\right) 2^{2n}}{\left( 2n\right)
!}|B_{2n}|\left( p^{2n-1}-\frac{1}{2^{2n}-1}\right) x^{2n-1} \\
&:&=\sum_{n=2}^{\infty }c_{n}d_{n}x^{2n-1},
\end{eqnarray*}%
where 
\begin{eqnarray*}
c_{n} &=&\frac{\left( 2^{2n}-1\right) 2^{2n}}{\left( 2n\right) !}|B_{2n}|%
\tfrac{p^{2n-1}-\frac{1}{2^{2n}-1}}{p-\left( \frac{1}{2^{2n}-1}\right)
^{1/\left( 2n-1\right) }}>0, \\
d_{n} &=&p-h\left( n\right) \text{, \ \ \ }h\left( n\right) =\left( \tfrac{1%
}{2^{2n}-1}\right) ^{1/\left( 2n-1\right) }
\end{eqnarray*}%
for $n\geq 1$ and $p\in (0,1]$.

Considering the function $g:\left( 1/2,\infty \right) \mapsto \left(
0,\infty \right) $ defined by 
\begin{equation}
g\left( x\right) =\left( \frac{1}{2^{2x}-1}\right) ^{1/\left( 2x-1\right) },
\label{h(x)}
\end{equation}%
and differentiation leads to%
\begin{eqnarray*}
\frac{2\left( 2x-1\right) ^{2}}{g\left( x\right) }g^{\prime }\left( x\right)
&=&\ln \left( 2^{2x}-1\right) -\frac{\left( 2x-1\right) 2^{2x}\ln 2}{\left(
2^{2x}-1\right) }:=g_{1}\left( x\right) , \\
g_{1}^{\prime }\left( x\right) &=&\frac{\allowbreak 2^{2x+1}\ln ^{2}2}{%
\left( 2^{2x}-1\right) ^{2}}\left( 2x-1\right) >0,
\end{eqnarray*}%
which reveals that $g_{1}$ is increasing on $\left( 1/2,\infty \right) $,
and therefore $g_{1}\left( x\right) >g_{1}\left( 1/2^{+}\right) =0$. It
follows that $g^{\prime }\left( x\right) >0$, then, $g$ is increasing on $%
\left( 1/2,\infty \right) $, hence for $n\geq 1$ 
\begin{equation*}
\tfrac{1}{3}=g\left( 1\right) \leq g\left( n\right) \leq g\left( \infty
\right) =\lim_{n\rightarrow \infty }\left( \tfrac{1}{2^{2n}-1}\right)
^{1/\left( 2n-1\right) }=\allowbreak \tfrac{1}{2},
\end{equation*}%
and then, $d_{n}=p-g\left( n\right) \leq 0$ if $p\in (0,1/3]$ and $%
d_{n}=p-g\left( n\right) \geq 0$ if $p\in \left[ 1/2,1\right] $. Thus, if $%
p\in (0,1/3]$ then $f_{p}^{\prime }\left( x\right) <0$, that is, $f_{p}$ is
decreasing, and it is derived that%
\begin{equation*}
\ln \beta _{p}\left( \tfrac{\pi }{2}\right) =f_{p}\left( \tfrac{\pi }{2}%
^{-}\right) <f_{p}\left( x\right) <\lim_{x\rightarrow 0^{+}}f_{p}\left(
x\right) =0,
\end{equation*}%
which yields (\ref{Mb}).

Likewise, if $p\in \left[ 1/2,1\right] $ then $f_{p}^{\prime }\left(
x\right) >0$, then, $f_{p}$ is increasing, and (\ref{Mb}) is reversed, which
completes the proof.
\end{proof}

\section{Corollaries}

Putting $p=1/3$\ in Theorem \ref{Theorem Mb}, we have

\begin{corollary}
\label{Corollary Mb1}(i) For $x\in \left( 0,\pi /2\right) $%
\begin{equation}
\beta _{1/3}\left( \tfrac{\pi }{2}\right) \cos ^{3}\frac{x}{3}<\frac{\sin x}{%
x}<\cos ^{3}\frac{x}{3}  \label{Mb1}
\end{equation}%
with the best constants $\beta _{1/3}\left( \pi /2\right) =16\sqrt{3}/\left(
9\pi \right) =\allowbreak 0.980\,14...$ and $1$.

(ii) For $x\in \left( 0,\pi /4\right) $ the inequalities 
\begin{equation}
\beta _{1/3}^{\left( 1\right) }\left( \tfrac{\pi }{4}\right) \cos ^{3}\frac{x%
}{3}<\frac{\sin x}{x}<\cos ^{3}\frac{x}{3}  \label{Mb2}
\end{equation}%
hold, where $\beta _{1/3}\left( \pi /4\right) =16\left( 3\sqrt{3}-5\right)
/\pi =\allowbreak 0.999\,00...$ and $1$ are the best possible.
\end{corollary}

\begin{remark}
The second inequality of (\ref{Mb1}) holds for $x\in \left( 0,3\pi /2\right) 
$. In fact, we define%
\begin{equation*}
h\left( x\right) :=x-\left( \cos \frac{x}{3}\right) ^{-3}\sin x
\end{equation*}%
and differentiation yields 
\begin{equation*}
h^{\prime }\left( x\right) =1-\frac{\sin x}{\cos ^{4}\frac{x}{3}}\sin \frac{x%
}{3}-\frac{\cos x}{\cos ^{3}\frac{x}{3}}=\tan ^{4}\frac{x}{3}\allowbreak >0.
\end{equation*}%
Hence for $x\in \left( 0,3\pi /2\right) $ then $h\left( x\right) >h\left(
0\right) =0$, which implies that the second inequality of (\ref{Mb1}) holds
for $x\in \left( 0,3\pi /2\right) $.
\end{remark}

Thus, by replacing $x$ for $3x$ the second inequality of (\ref{Mb1}) and
next using duplication formula for sine function we have

\begin{corollary}
\label{Corollary Mb2}For $x\in \left( 0,\pi /2\right) $ 
\begin{equation}
\frac{\sin x}{x}\left( 4\cos ^{2}x-1\right) =\frac{\sin 3x}{x}<3\cos ^{3}x
\label{Mb3}
\end{equation}%
holds.
\end{corollary}

Utilizing the second inequality of (\ref{Mb1}) holds for $x\in \left( 0,3\pi
/2\right) $ we also obtain

\begin{corollary}
\label{Corollary Mb4}For $x\in \left( 0,\pi /2\right) $ we have 
\begin{equation}
\frac{\sin x}{x}<\allowbreak \frac{3}{4}\frac{\left( \cos x+1\right) ^{2}}{%
2\cos x+1}.  \label{Mb4}
\end{equation}
\end{corollary}

\begin{proof}
Since the second inequality of (\ref{Mb1}) holds for $x\in \left( 0,3\pi
/2\right) $, we have 
\begin{equation*}
\frac{\sin \left( 3x/2\right) }{3x/2}<\cos ^{3}\frac{x}{2}.
\end{equation*}%
Multiplying both sides by $\cos \left( x/2\right) $ and next using "product
into sum" and half-angle formulas give us 
\begin{equation}
\frac{\sin 2x+\sin x}{3x}<\allowbreak \left( \frac{\cos x+1}{2}\right) ^{2},
\label{Mb41}
\end{equation}%
which, after applying double-angle formula and next dividing both sides by $%
\left( 2\cos x+1\right) /3$, is the desired inequality.
\end{proof}

Putting $p=1/2\,$in Theorem \ref{Theorem Mb}, we get

\begin{corollary}
\label{Corollary Mb5}For $x\in \left( 0,\pi /2\right) $, the inequalities 
\begin{equation}
\left( \cos \frac{x}{2}\right) ^{2}<\frac{\sin x}{x}<\beta _{1/2}\left( 
\tfrac{\pi }{2}\right) \left( \cos \frac{x}{2}\right) ^{2}  \label{Mb5}
\end{equation}%
hold, where $\beta _{1/2}\left( \pi /2\right) =\allowbreak 4/\pi $ and $1$
are the best possible constants.
\end{corollary}

\begin{remark}
The first inequality of (\ref{Mb5}) or (\ref{Qi}) is equivalent to $\tan
\left( x/2\right) >x/2$, which holds for $x/2\in \left( 0,\pi /2\right) $,
that is, $x\in \left( 0,\pi \right) $. The second one of (\ref{Mb1}) or (\ref%
{Klen}) holds for $x\in \left( 0,3\pi /2\right) $ due to Remark previous.
While the third one of (\ref{Klen}) holds for $x\in \left( -\infty ,\infty
\right) $, since 
\begin{equation*}
\cos ^{3}\frac{x}{3}-\frac{2+\cos x}{3}=-\frac{1}{3}\left( \cos \frac{x}{3}%
+2\right) \left( \cos \frac{x}{3}-1\right) ^{2}<0.
\end{equation*}%
Consequently, the value range of variable $x$ such that (\ref{Klen}) holds
can be extended to $\left( 0,\pi \right) $, which slightly improves (\ref%
{Klen}).$\allowbreak $
\end{remark}

\section{Applications}

As consequences of main results, we will establish some new inequalities in
this section. The following is a direct corollary of Theorem \ref{Theorem Mb}%
.

\begin{corollary}
We have 
\begin{equation}
\left( \frac{2}{\pi }\right) ^{p}\frac{1}{p}\tan \frac{p\pi }{2}%
<\int_{0}^{\pi /2}\left( \frac{\sin x}{x}\right) ^{p}dx<\int_{0}^{\pi
/2}\left( \cos px\right) =\frac{1}{p}\sin \frac{p\pi }{2}  \label{A1}
\end{equation}%
if $p\in (0,1/3]$. Inequalities (\ref{A1}) is reversed if $p\in \left[ 1/2,1%
\right] $.
\end{corollary}

For the estimate for the sine integral defined by 
\begin{equation*}
\func{Si}\left( x\right) =\int_{0}^{x}\frac{\sin t}{t}dt,
\end{equation*}%
there has some results, for example, Qi \cite{Qi.12(4)(1996)} showed that 
\begin{equation*}
\allowbreak 1.\,\allowbreak 333\,3...=\frac{4}{3}<\func{Si}\left( \frac{\pi 
}{2}\right) <\frac{\pi +1}{3}=\allowbreak 1.\,\allowbreak 380\,5...\text{;}
\end{equation*}%
the following two estimations are due to Wu \cite{Wu.19(12)(2006)}, \cite%
{Wu.12(2)(2008)}: 
\begin{eqnarray*}
\allowbreak 1.\,\allowbreak 356\,9... &=&\frac{\pi +5}{6}<\func{Si}\left( 
\frac{\pi }{2}\right) <\frac{\pi +1}{3}=\allowbreak 1.\,\allowbreak 380\,5...%
\text{,} \\
1.\,\allowbreak 368\,8... &=&\frac{92-\pi ^{2}}{60}<\func{Si}\left( \frac{%
\pi }{2}\right) <\frac{8+4\pi }{15}=\allowbreak 1.\,\allowbreak 371\,1...%
\text{.}
\end{eqnarray*}%
Now we give a more better one.$\allowbreak $

\begin{corollary}
We have%
\begin{equation}
\allowbreak 1.\,\allowbreak 369\,6...=\tfrac{\left( 3\sqrt{3}-5\right)
\left( 2\pi +9\sqrt{3}+22\right) }{2\pi }<\func{Si}\left( \tfrac{\pi }{2}%
\right) <\tfrac{2\pi +9\sqrt{3}+22}{32}=\allowbreak 1.\,\allowbreak
371\,0....  \label{A2}
\end{equation}
\end{corollary}

\begin{proof}
By \ref{Corollary Mb2}, for $x\in \left( 0,\pi /2\right) $ we have%
\begin{equation*}
\beta _{1/3}\left( \tfrac{\pi }{4}\right) \cos ^{3}\frac{x}{6}<\frac{\sin 
\frac{x}{2}}{\frac{x}{2}}<\cos ^{3}\frac{x}{6},
\end{equation*}%
where $\beta _{1/3}\left( \pi /4\right) =16\left( 3\sqrt{3}-5\right) /\pi $,
and multiplying both sides by $\cos \left( x/2\right) $ leads to 
\begin{equation}
\beta _{1/3}\left( \tfrac{\pi }{4}\right) \cos ^{3}\frac{x}{6}\cos \frac{x}{2%
}<\frac{\sin x}{x}=\frac{\sin \frac{x}{2}}{\frac{x}{2}}\cos \frac{x}{2}<\cos
^{3}\frac{x}{6}\cos \frac{x}{2}.  \label{5.1}
\end{equation}%
Integrating both sides over $\left[ 0,\pi /2\right] $ yields 
\begin{equation*}
\beta _{1/3}\left( \tfrac{\pi }{4}\right) \int_{0}^{\pi /2}\cos ^{3}\frac{x}{%
6}\cos \frac{x}{2}dx<\int_{0}^{\pi /2}\frac{\sin x}{x}dx<\int_{0}^{\pi
/2}\cos ^{3}\frac{x}{6}\cos \frac{x}{2}dx.
\end{equation*}%
From the following%
\begin{equation*}
\int_{0}^{\pi /2}\cos \frac{x}{2}\cos ^{3}\frac{x}{6}dx=\frac{2\pi +9\sqrt{3}%
+22}{32}
\end{equation*}%
(\ref{A1}) follows.
\end{proof}

The Catalan constant \cite{Catalan.31(7)} 
\begin{equation*}
K=\sum_{n=0}^{\infty }\frac{\left( -1\right) ^{n}}{\left( 2n+1\right) ^{2}}%
=0.9159655941772190...
\end{equation*}%
is a famous mysterious constant appearing in many places in mathematics and
physics. Its integral representations contain the following \cite%
{Bradley.2001}%
\begin{equation*}
K=\int_{0}^{1}\frac{\arctan x}{x}dx=\frac{1}{2}\int_{0}^{\pi /2}\frac{x}{%
\sin x}dx.
\end{equation*}%
We present two estimations for $K$ below.

\begin{corollary}
We have 
\begin{eqnarray}
\allowbreak \tfrac{1}{8}\left( 4+3\ln 3\right) &<&K<\tfrac{3\sqrt{3}\pi }{128%
}\left( 4+3\ln 3\right) ,  \label{A31} \\
\tfrac{8\sqrt{3}}{9}\ln \left( 1+\sqrt{3}\right) -\tfrac{40-22\sqrt{3}}{3}
&<&K<\pi \allowbreak \left( \tfrac{9+5\sqrt{3}}{36}\ln \left( \sqrt{3}%
+1\right) -\tfrac{1+5\sqrt{3}}{48}\right) .  \label{A32}
\end{eqnarray}
\end{corollary}

\begin{proof}
By inequalities (\ref{Mb1}) we get for $\left( 0,\pi /2\right) $ 
\begin{equation*}
\frac{1}{\cos ^{3}\tfrac{x}{3}}<\frac{x}{\sin x}<\frac{9\pi }{16\sqrt{3}}%
\frac{1}{\cos ^{3}\tfrac{x}{3}},
\end{equation*}%
and integrating both sides over $\left[ 0,\pi /2\right] $ yields%
\begin{equation*}
\int_{0}^{\pi /2}\frac{dx}{\cos ^{3}\frac{x}{3}}<\int_{0}^{\pi /2}\frac{x}{%
\sin x}dx<\int_{0}^{\pi /2}\frac{dx}{\cos ^{3}\frac{x}{3}}.
\end{equation*}%
From the 
\begin{equation*}
\int_{0}^{\pi /2}\frac{dx}{\cos ^{3}\frac{x}{3}}=1+\allowbreak \frac{3\ln 3}{%
4}
\end{equation*}%
(\ref{A31}) follows.

Now we prove (\ref{A32}). From (\ref{5.1}) it is deduced that 
\begin{equation*}
\frac{1}{\cos \frac{x}{2}\cos ^{3}\frac{x}{6}}<\frac{x}{\sin x}<\frac{1}{%
\beta _{1/3}\left( \pi /4\right) }\frac{1}{\cos \frac{x}{2}\cos ^{3}\frac{x}{%
6}},
\end{equation*}%
and integrating over $\left[ 0,\pi /2\right] $ yields$\bigskip $%
\begin{equation*}
\int_{0}^{\pi /2}\frac{dx}{\cos \frac{x}{2}\cos ^{3}\frac{x}{6}}%
<\int_{0}^{\pi /2}\frac{x}{\sin x}dx<\frac{1}{\beta _{1/3}\left( \pi
/4\right) }\int_{0}^{\pi /2}\frac{dx}{\cos \frac{x}{2}\cos ^{3}\frac{x}{6}}.
\end{equation*}%
Direct computation gives 
\begin{eqnarray*}
\int \frac{dt}{\cos 3t\cos ^{3}t}\allowbreak &=&\tfrac{\sin t}{576\cos ^{5}t}%
-\tfrac{559\sin t}{1152\cos ^{3}t}-\tfrac{113\sin t}{384\cos t}+\tfrac{%
65\sin ^{3}t}{192\cos ^{3}t}+\tfrac{37\sin ^{3}t}{1152\cos ^{5}t} \\
&&-\allowbreak \tfrac{13\sin ^{5}t}{384\cos ^{5}t}+\allowbreak \tfrac{4\sqrt{%
3}}{27}\ln \tfrac{1+\sin \left( 2t-\frac{\pi }{6}\right) }{1-\sin \left( 2t+%
\frac{\pi }{6}\right) }+C,
\end{eqnarray*}%
where $C$ is a constant. Hence, 
\begin{equation*}
\int_{0}^{\pi /2}\frac{dx}{\cos \frac{x}{2}\cos ^{3}\frac{x}{6}}%
=6\int_{0}^{\pi /12}\frac{dt}{\cos 3t\cos ^{3}t}=\frac{8\sqrt{3}}{9}\left(
2\ln \left( \sqrt{3}+1\right) -10\sqrt{3}+\frac{33}{2}\right) ,
\end{equation*}%
it follows that (\ref{A32}) holds.
\end{proof}

The Schwab-Borchardt mean of two numbers $a\geq 0$ and $b>0$, denoted by $%
SB\left( a,b\right) $, is defined as \cite[Theorem 8.4]{Biernacki.9.1955}, 
\cite[3, (2.3)]{Carlson.78(1971)}%
\begin{equation*}
SB\left( a,b\right) =\left\{ 
\begin{array}{cc}
\frac{\sqrt{b^{2}-a^{2}}}{\arccos \left( a/b\right) } & \text{if }a<b, \\ 
a & \text{if }a=b, \\ 
\frac{\sqrt{a^{2}-b^{2}}}{\func{arccosh}\left( a/b\right) } & \text{if }a>b.%
\end{array}%
\right.
\end{equation*}%
The properties and certain inequalities involving Schwab-Borchardt mean can
be found in \cite{Neuman.14(2003)}, \cite{Neuman.JMI.2012.inprint}. We now
establish a new inequality for this mean.

\begin{corollary}
For $t>0$, we have 
\begin{equation}
\left( 4t^{2}-1\right) SB\left( t,1\right) \leq 3t^{3}.  \label{A4}
\end{equation}
\end{corollary}

\begin{proof}
For $t\in \left( 0,1\right) $, letting $\cos x=t$ in (\ref{Mb2}) we get 
\begin{equation*}
\frac{\sqrt{1-t^{2}}}{\arccos t}\left( 4t^{2}-1\right) <3t^{3}.
\end{equation*}%
For $t\in \left( 1,\infty \right) $, we use Lin's inequality \cite%
{Lin.79.1972} of positive numbers $a,b>0$ with $a\neq b$ 
\begin{equation*}
\frac{a-b}{\ln a-\ln b}<\left( \frac{a^{1/3}+b^{1/3}}{2}\right) ^{3}
\end{equation*}%
to deduce 
\begin{equation*}
\frac{\sinh u}{u}<\cosh ^{3}\frac{u}{3}
\end{equation*}%
by setting $u=\ln \sqrt{a/b}$. Letting $\cosh \left( u/3\right) =t$ and next
using duplication formula for sinh function lead us to 
\begin{equation*}
\frac{\sqrt{t^{2}-1}}{\func{arccosh}t}\left( 4t^{2}-1\right) <3t^{3},
\end{equation*}%
which proves the desired inequality.
\end{proof}

The Seiffert's mean \cite{Seiffert.42(1987)} of positive numbers $a,b>0$
with $a\neq b$ is defined by 
\begin{equation*}
P=P\left( a,b\right) =\frac{a-b}{2\arcsin \frac{a-b}{a+b}},
\end{equation*}%
and let $A=A\left( a,b\right) $ and $G=G\left( a,b\right) $ denote the
arithmetic mean and geometric mean. With $x=\arcsin \frac{a-b}{a+b}\in
\left( 0,\pi /2\right) $, we have 
\begin{equation*}
\frac{P}{A}=\frac{\sin x}{x}\text{, \ \ \ }\frac{G}{A}=\cos x.
\end{equation*}%
Thus Corollary \ref{Corollary Mb4} can be restated as follows.

\begin{corollary}
For $a,b>0$ with $a\neq b$, we have 
\begin{equation}
P<\allowbreak \frac{3}{4}\frac{\left( A+G\right) ^{2}}{2G+A}.  \label{A5}
\end{equation}
\end{corollary}


\begin{thebibliography}{99}
\bibitem{Anderson.New York. 1997} G. D. Anderson, M. K. Vamanamurthy \textsc{%
and} M. Vuorinen, \emph{Conformal Invariants, Inequalities, and
Quasiconformal Maps}, New York 1997.

\bibitem{Biernacki.9.1955} M. Biernacki \textsc{and} J. Krzyz, On the
monotonicity of certain functionals in the theory of analytic functions, 
\emph{Annales Universitatis Mariae Curie-Sklodowska} \textbf{9} (1995)
135--147.

\bibitem{Bradley.2001} D. M. Bradley, Representations of Catalan's constant,
2001, available online at
http://citeseerx.ist.psu.edu/viewdoc/summary?doi=10.1.1.26.1879.

\bibitem{Borwein.1987} J. M. Borwein \textsc{and} P. B. Bowein, \emph{Pi and
the AGM: A Study in Analytic Number Theory and Computational Complexity},
John Wiley and Sons, New York, 1987.

\bibitem{Catalan.31(7)} E. Catalan, Recherches sur la constante $G$, et sur
les int\'{e}grales eul\'{e}riennes, \emph{M\'{e}oires de l'Academie
imperiale des sciences de Saint-P\'{e}tersbourg, Ser.} \textbf{31} (7)
(1883).

\bibitem{Carlson.78(1971)} B. C. Carlson, Algorithms involving arithmetic
and geometric means, \emph{Amer. Math. Monthly} \textbf{78} (1971), 496--505.

\bibitem{Handbook.math.1979} Group of compilation, Handbook of Mathematics,
Peoples' Education Press, Beijing, China, 1979. (Chinese)

\bibitem{Klen.JIA.2010} R. Kl\'{e}n, M. Visuri \textsc{and} M. Vuorinen, On
Jordan type inequalities for hyperbolic functions, \emph{J. Inequal. Appl.} 
\textbf{2010} (2010), Art. ID 362548, 14 pages, doi:10.1155/2010/362548.

\bibitem{Lin.79.1972} T. P. Lin, The power mean and the logarithmic mean, 
\emph{Amer. Math. Monthly} \textbf{81} (1974), 879--883.

\bibitem{Mitrinovic.Springer.1970} D. S. Mitrinovi\'{c} \textsc{and} P. M.
Vasi\'{c}, \emph{Analytic Inequalities}, Springer, New York, 1970.

\bibitem{Neuman.14(2003)} E. Neuman \textsc{and} J. S\'{a}ndor, On the
Schwab-Borchardt mean, \emph{Math. Pannon.} \textbf{14} (2003), 253--266.

\bibitem{Neuman.JMI.2012.inprint} E. Neuman, Inequalities for the
Schwab-Borchardt mean and their applications., \emph{J. Math. Inequal.} in
print.

\bibitem{Qi.12(4)(1996)} F. Qi, Extensions and sharpenings of Jordan's and
Kober's inequality,\ J\emph{ournal of Mathematics for Technology} \textbf{12}
(4) (1996), 98--102. (Chinese)

\bibitem{Qi.JIA.2009} F. Qi, D.-W. Niu \textsc{and} B.-N. Guo, Refinements,
generalizations, and applications of Jordan's inequality and related
problems, \emph{J. Inequal. Appl.} \textbf{2009} (2009), Art. ID 271923, 52
pages.

\bibitem{Qi.MIA.2(4)(1999)} F. Qi, L.-H. Cui \textsc{and} S.-L. Xu, Some
inequalities constructed by Tchebysheff's integral inequality, \emph{Math.
Inequal. Appl.} \textbf{2} (4) (1999) 517--528.

\bibitem{Seiffert.42(1987)} H.-J. Seiffert, Werte zwischen dem geometrischen
und dem arithmetischen Mittel zweier Zahlen, \emph{Elem. Math.} \textbf{42}
(1987), 105-107.

\bibitem{Vamanamurthy.183.1994} M. K. Vamanamurthy \textsc{and} M. Vuorinen,
Inequalities for means, \emph{J. Math. Anal. Appl.} \textbf{183 }(1994)
155--166.

\bibitem{Wu.19(12)(2006)} Sh.-H. Wu \textsc{and} L. Debnath, A new
generalized and sharp version of Jordan's inequality and its applications to
the improvement of the Yang Le inequality,\ \emph{Appl. Math. Letters} 
\textbf{19} (12) (2006), 1378--1384.

\bibitem{Wu.12(2)(2008)} Sh.-H. Wu, Sharpness and generelization of Jordan's
inequality and its application, \emph{Taiwanese J. Math.} \textbf{12} (2)
(2008), 325-336.
\end{thebibliography}
\end{document}